\documentclass[a4paper,12pt]{article}
\usepackage{hyperref}
\hypersetup{
    colorlinks=true,
    linktoc=all,
    linkcolor=red,     %choose some color if you want links to stand out
    citecolor=blue,     %choose some color if you want links to stand out
}
\usepackage{graphicx}
\usepackage{amsmath,amssymb,amsfonts,amsthm,graphics,enumitem}
\usepackage{latexsym, bm}
\usepackage{multicol}
\usepackage{indentfirst}

\usepackage{setspace}
\newtheorem{theorem}{Theorem}[section]
\newtheorem{lemma}{Lemma}[section]

\newtheorem{corollary}{Corollary}[section]

\newtheorem{problem}{Problem}[section]

\newtheorem{conjecture}{Conjecture}[section]

\newtheorem{observation}{Observation}[section]
\newcommand{\ignore}[1]{}
\newtheorem{claimnum}{Claim}
\begin{document}
\noindent
\begin{spacing}{1}

\title {Edge and spectral conditions for rainbow pancyclicity in graph collections}
\date{}

\author{
Lihua You\footnote{School of Mathematical Sciences, South China Normal University, Guangzhou, 510631, P. R. China.
E-mail: {\tt ylhua@scnu.edu.cn}.} \;\;   \;\; Xiaoxue Zhang\footnote{Corresponding author. School of Mathematical Sciences, South China Normal University, Guangzhou, 510631, P. R. China.
E-mail: {\tt zhang\_xx1209@163.com}.}\;\;   \;\; Xinghui Zhao\footnote{School of Mathematical Sciences, South China Normal University, Guangzhou, 510631, P. R. China.
E-mail: {\tt xhzhao@m.scnu.edu.cn}.}
}
\maketitle
\begin{abstract}
Let $\mathbf{G}=\{G_1,\dots,G_{n}\}$ be a collection of not necessarily distinct $n$-vertex graphs with a common vertex set $V$. A cycle $C$ with $V(C)\subseteq V$ and $|E(C)|\leq n$ is called \emph{rainbow} in $\mathbf{G}$, if there exists an injection $\phi\colon E(C)\to [n]$ such that $e\in E(G_{\phi(e)})$ for each $e\in E(C)$.
The graph collection $\mathbf{G}$ is said to be \emph{rainbow pancyclic} if it contains a rainbow cycle of every length from 3 to $n$.
In this paper, we show that if $e(G_i)\ge \binom{n-1}{2}+1$ for each $i\in[n]$ with $n\ge 3$, then $\mathbf{G}$ is rainbow pancyclic, apart from three explicitly described exceptional graph collections. This answers Problem $1$ of [Discrete Math., \textbf {348}(2025), 114600] and strengthens the result from rainbow Hamiltonicity to rainbow pancyclicity. As a consequence, we obtain that if $\rho(G_i)>n-2$ for each $i\in[n]$, then $\mathbf{G}$ is rainbow pancyclic unless $G_1=G_2=\dots=G_n\cong K_1\vee(K_{n-2}\cup K_1)$, which improves Theorem $5$ of [Discrete Math., \textbf {348}(2025), 114600]. We also characterize all graph collections that are not rainbow pancyclic under the condition $\rho(G_i)\ge n-2$ for each $i\in[n]$.


\end{abstract}

\section{Introduction}

\subsection{Hamiltonicity and pancyclicity}
A graph is Hamiltonian if it contains a cycle covering all its vertices exactly once. It is natural to ask what conditions guarantee Hamiltonicity for a given graph. The study of Hamiltonicity has a long history and has attracted the interest of many researchers. A number of results have been established. A classical result by Dirac \cite{Dirac} states that every $n$-vertex graph with minimum degree at least $\frac{n}{2}$ is Hamiltonian. In 1961, Ore \cite{Ore} proved a sufficient condition on the number of edges for a graph to be Hamiltonian. Furthermore, Ore \cite{Ore} and Bondy \cite{Bondy1972} characterized the exceptional non-Hamiltonian graphs with exactly $\binom{n-1}{2}+1$ edges, respectively.

\begin{theorem}[\upshape\cite{Ore}]\label{thm1-2}
A graph with $e(G)> \binom{n-1}{2}+1$ is Hamiltonian.
\end{theorem}

\begin{theorem}[\upshape\cite{Bondy1972,Ore}]\label{thm1-3}
Let $G$ be a non-Hamiltonian graph of order $n$ with $e(G)=\binom{n-1}{2}+1$. Then either $G\cong K_1\vee(K_{n-2}\cup K_1)$ or $G\cong K_2\vee 3K_1$.
\end{theorem}

In 2010, Fiedler and Nikiforov \cite{Fiedler} provided a spectral condition guaranteeing the Hamiltonicity of graphs. This result follows readily by applying Stanley's inequality $\rho^2(G)+\rho(G)\le 2e(G)$ together with Theorems \ref{thm1-2} and \ref{thm1-3}.

\begin{theorem}[\upshape\cite{Fiedler}]\label{pro1-1}
Let $G$ be an $n$-vertex graph and $\rho(G)> n-2$.
Then $G$ is Hamiltonian unless $G\cong K_1\vee(K_{n-2}\cup K_1)$.
\end{theorem}

A natural and much stronger notion than Hamiltonicity is that of pancyclicity.
An $n$-vertex graph with $n\ge 3$ is said to be pancyclic if it contains a cycle of every length from 3 to $n$.
In 1971, Bondy \cite{BondyJCTB} strengthened Dirac's result by showing that the same minimum degree condition ensures pancyclicity except for $K_{\frac{n}{2},\frac{n}{2}}$. In fact, he established a stronger result.

\begin{theorem}[\upshape\cite{BondyJCTB}]\label{thm1-1}
Let $G$ be a Hamiltonian graph of order $n$ with $e(G)\ge \frac{n^2}{4}$. Then $G$ is either pancyclic or $G\cong K_{\frac{n}{2},\frac{n}{2}}$.
\end{theorem}

Later, Bondy \cite{Bondy1971} posed the following meta-conjecture.
\begin{conjecture}[\upshape\cite{Bondy1971}]\label{con1-1}
Almost any nontrivial condition on a graph that implies Hamiltonicity also implies pancyclicity, with at most a small family of exceptional graphs.
\end{conjecture}

Inspired by Conjecture \ref{con1-1}, we obtain the following Theorem \ref{thm1-4} by Theorems \ref{thm1-2}-\ref{thm1-1} immediately.

\begin{theorem}\label{thm1-4}
Let $G$ be an $n$-vertex graph with $n\ge3$.
\begin{enumerate}[label=(\roman*), font=\upshape, itemsep=2pt, align=left, leftmargin=1em]
 \item If $e(G)\ge \binom{n-1}{2}+1$, then $G$ is pancyclic unless $G\in\{K_1\vee(K_{n-2}\cup K_1),K_2\vee 3K_1,K_{2,2}\}$.
 \item If $\rho(G)> n-2$, then $G$ is pancyclic unless $G\cong K_1\vee(K_{n-2}\cup K_1)$.
\end{enumerate}
\end{theorem}

\begin{proof}
For (i), we proceed by considering whether $G$ is Hamiltonian.

If $G$ is not Hamiltonian, then $e(G)\le \binom{n-1}{2}+1$ by Theorem \ref{thm1-2}, and thus $e(G)=\binom{n-1}{2}+1$ since $e(G)\ge \binom{n-1}{2}+1$. It follows that $G\cong K_1\vee(K_{n-2}\cup K_1)$ or $G\cong K_2\vee 3K_1$ by Theorem \ref{thm1-3}.

If $G$ is Hamiltonian, then the case $n=3$ is immediate. For $n\ge 4$, $\binom{n-1}{2}+1-\frac{n^2}{4}=\frac{(n-2)(n-4)}{4}\ge 0$, which implies $e(G)\ge \frac{n^2}{4}$. By Theorem \ref{thm1-1}, $G$ is pancyclic or $G\cong K_{\frac{n}{2},\frac{n}{2}}$. When $G\cong K_{\frac{n}{2},\frac{n}{2}}$, we have $e(G)=\frac{n^2}{4}$, which implies $n=4$ since $e(G)\ge \binom{n-1}{2}+1$ and $n\ge 4$, say, $G\cong K_{2,2}$.

Therefore, $G$ is pancyclic unless $G\in\{K_1\vee(K_{n-2}\cup K_1),K_2\vee 3K_1,K_{2,2}\}$.

For (ii), we have $e(G)\ge \binom{n-1}{2}+1$ since $\rho(G)> n-2$ and $\rho^2(G)+\rho(G)\le 2e(G)$. By (i), $G$ is pancyclic unless $G\in\{K_1\vee(K_{n-2}\cup K_1),K_2\vee 3K_1,K_{2,2}\}$. It is easy to obtain that $\rho(K_1\vee(K_{n-2}\cup K_1))>\rho(K_{n-1}\cup K_1)=n-2$, $\rho(K_2\vee 3K_1)=3=n-2$ and $\rho(K_{2,2})=2=n-2$. Therefore, $G$ is pancyclic unless $G\cong K_1\vee(K_{n-2}\cup K_1)$.

Therefore, this completes the proof.
\end{proof}

\subsection{Main results}
In recent years, many classical theorems have been investigated in the context of graph collections, yielding their rainbow analogues. Let $\mathbf{G}=\{G_1,\dots,G_s\}$ be a collection of not necessarily distinct graphs with a common vertex set $V$.
For a graph $H$ with $V(H)\subseteq V$ and $|E(H)|\leq s$, if there exists an injection $\phi\colon E(H)\to [s]$ such that $e\in E(G_{\phi(e)})$ for each $e\in E(H)$, in other words, each edge of $H$ comes from a distinct graph $G_i$ with $i\in[s]$, then we say that $H$ is \emph{rainbow} in $\mathbf{G}$. A collection of $n$-vertex graphs with a common vertex set is said to be \emph{rainbow pancyclic} if it contains a rainbow cycle of every length from 3 to $n$.

More work on graph collections is already available, including general graph collections \cite{Cheng2026,Cheng,Guo,Guo2026,He,Joos,Li2023,Li2024,LiandWang,LiuandChen,MaandZhang,Sun,Zhang2026}, bipartite graph collections \cite{Bradshaw,Chen,Chen2026,He2026,Hu,Ma}, directed graph collections \cite{Babinski,Chakraborti,Gerbner,Li2025}, hypergraph collections \cite{Cheng2023,Cheng2025,Gao2022,Lu2022,Lu2021} and so on.

In 2020, Joos and Kim \cite{Joos} proved that if $\mathbf{G}=\{G_1,\dots, G_n\}$ is a collection of $n$ graphs with a common vertex set $V$ with $|V|=n$ such that $\delta(G_i)\geq\frac{n}{2}$ for each $i\in[n]$, then there exists a rainbow Hamiltonian cycle in $\mathbf{G}$. This result can be regarded as a generalization of Dirac's theorem, since it reduces to Dirac's theorem when $G_1=\dots=G_n$.
In line with Conjecture \ref{con1-1}, Li, Li and Li \cite{Li2024} showed that the same degree condition further guarantees rainbow pancyclicity of $\mathbf{G}$ unless $G_1=\dots=G_n\cong K_{\frac{n}{2},\frac{n}{2}}$.

In 1963, Moon and Moser \cite{Moon} established a Dirac-type condition for Hamiltonicity in balanced bipartite graphs. In 1982, Schmeichel and Mitchem \cite{Schmeichel} further extended this result to bipancyclicity. In 2021, Bradshaw \cite{Bradshaw} generalized both classical results above to graph collections and obtained their corresponding rainbow analogues.
Recently, Liu, Chen and Ma \cite{LiuandChen} established a stronger Ore-type condition for the existence of rainbow Hamiltonian cycles in graph collections. Under the same Ore-type condition, Li, Wang and Yan \cite{LiandWang} further considered the rainbow pancyclicity.

Overall, these results lend further support to Conjecture \ref{con1-1} in the context of graph collections. For more results, we refer the readers to the survey \cite{SunandWang}.

In 2025, Zhang and van Dam \cite{Zhang} considered rainbow versions of Theorem \ref{thm1-2} and Theorem \ref{pro1-1}, and obtained the following results.

\begin{theorem}[\upshape\cite{Zhang}]\label{thm1-5}
Let $n\ge 4$ and $\mathbf{G}=\{G_1, \dots, G_{n}\}$ be a collection of not necessarily distinct $n$-vertex graphs with a common vertex set $V$.
\begin{enumerate}[label=(\roman*), font=\upshape, itemsep=2pt, align=left, leftmargin=1em]
 \item If $e(G_i)> \binom{n-1}{2}+1$ for each $i\in [n]$, then $\mathbf{G}$ contains a rainbow Hamiltonian cycle.
 \item If $\rho(G_i)> n-2$ for each $i\in [n]$, then $\mathbf{G}$ contains a rainbow Hamiltonian cycle unless $G_1=G_2=\dots=G_n\cong K_1\vee(K_{n-2}\cup K_1)$.
\end{enumerate}
\end{theorem}

In \cite{Zhang}, they also posed the following problem based on Theorems \ref{thm1-2} and \ref{thm1-3}.

\begin{problem}[\upshape\cite{Zhang}]\label{prob1-1}
Let $n\ge 6$ and $\mathbf{G}=\{G_1, \dots, G_{n}\}$ be a collection of not necessarily distinct $n$-vertex graphs with a common vertex set $V$ satisfying $e(G_i)\ge \binom{n-1}{2}+1$ for each $i\in [n]$. Does $\mathbf{G}$ contain a rainbow Hamiltonian cycle unless $G_1=G_2=\dots=G_n\cong K_1\vee(K_{n-2}\cup K_1)$?
\end{problem}

In this paper, in order to answer Problem \ref{prob1-1}, we obtain Theorem \ref{thm1-7}, which generalizes (i) of Theorem \ref{thm1-4} from graphs to graph collections. Moreover, Theorem \ref{thm1-7} supports Conjecture \ref{con1-1} in the context of graph collections, answers Problem \ref{prob1-1} and strengthens the result from rainbow Hamiltonicity to rainbow pancyclicity.

\begin{theorem}\label{thm1-7}
Let $n\ge 3$ and $\mathbf{G}=\{G_1, \dots, G_{n}\}$ be a collection of not necessarily distinct $n$-vertex graphs with a common vertex set $V$. If $e(G_i)\ge \binom{n-1}{2}+1$ for each $i\in [n]$, then $\mathbf{G}$ is rainbow pancyclic unless one of the following holds:
\begin{enumerate}[label=(\roman*), font=\upshape, itemsep=2pt, align=left, leftmargin=1em]
 \item $G_1=G_2=\dots=G_n\cong K_1\vee(K_{n-2}\cup K_1)$;
 \item $n=5$ and $G_1=G_2=\dots=G_5\cong K_2\vee 3K_1$;
 \item $n=4$ and $G_1=G_2=G_3=G_4\cong K_{2,2}$.
\end{enumerate}
\end{theorem}

Next, by Theorem \ref{thm1-7} and Stanley's inequality, we obtain the following Theorem \ref{thm1-8}, which strengthens (ii) of Theorem \ref{thm1-5} from rainbow Hamiltonicity to rainbow pancyclicity and is the rainbow version of (ii) of Theorem \ref{thm1-4}.

\begin{theorem}\label{thm1-8}
Let $n\ge 3$ and $\mathbf{G}=\{G_1, \dots, G_{n}\}$ be a collection of not necessarily distinct $n$-vertex graphs with a common vertex set $V$. If $\rho(G_i)> n-2$ for each $i\in [n]$, then $\mathbf{G}$ is rainbow pancyclic unless $G_1=G_2=\dots=G_n\cong K_1\vee(K_{n-2}\cup K_1)$.
\end{theorem}

Finally, we characterize all graph collections that are not rainbow pancyclic under the condition $\rho(G_i)\ge n-2$ for each $i\in[n]$. For the convenience of stating the result, we introduce the following notation. For each $v\in V$, let $F_v$ ($\cong K_{n-1}\cup K_1$) denote the graph on vertex set $V$ of order $n$ such that $v$ is an isolated vertex and $V\setminus\{v\}$ induces a complete graph $K_{n-1}$. For distinct $u,v\in V$, let $A_{v,u}$ be the graph obtained from $F_v$ by adding the edge $vu$. Then $A_{v,u}\cong K_1\vee(K_{n-2}\cup K_1)$ and $v$ is the pendant vertex of $A_{v,u}$.

\begin{theorem}\label{thm1-9}
Let $n\ge 3$ and $\mathbf{G}=\{G_1, \dots, G_{n}\}$ be a collection of not necessarily distinct $n$-vertex graphs with a common vertex set $V$. If $\rho(G_i)\geq n-2$ for each $i\in [n]$, then $\mathbf{G}$ is not rainbow pancyclic if and only if one of the following holds:
\begin{enumerate}[label=(\roman*), font=\upshape, itemsep=2pt, align=left, leftmargin=1em]
 \item there exists a vertex $v\in V$ such that $|\{i\in[n]:G_i=F_v\}| \ge n-1$;
 \item there exist distinct $u,v\in V$ such that for each $i\in[n]$,
 \[G_i\in
 \begin{cases}
 \{F_v,F_w,A_{v,u}\}, \text{where } V=\{u,v,w\}, & \text{if } n=3;\\[4pt]
 \{F_v,A_{v,u}\}, & \text{if } n\ge 4.
 \end{cases}
 \]
 \item $n=5$ and $G_1=G_2=\dots=G_5\cong K_2\vee 3K_1$;
 \item $n=4$ and $G_1=G_2=G_3=G_4\cong K_{2,2}$.
\end{enumerate}
\end{theorem}

\subsection{Notation and organization}
Here, we introduce some basic notation and concepts in graph theory. All graphs considered in this paper are simple, finite and undirected.
Let $G=(V(G),E(G))$ be a simple graph with vertex set $V(G)$ and edge set $E(G)$, where $|E(G)|=e(G)$. The \emph{adjacency matrix} $A(G)=(a_{uv})_{u,v\in V(G)}$ of $G$ is a $|V(G)|\times |V(G)|$ matrix, where $a_{uv}=1$ if $uv\in E(G)$, and $a_{uv}=0$ otherwise. The maximum of the moduli of the eigenvalues of $A(G)$ is called the \emph{spectral radius} of $G$ and denoted by $\rho(G)$. The \emph{complement graph} of $G$, denoted by $\overline{G}$, is a graph with the same vertex set $V$ such that any two distinct vertices $u,v\in V$ are adjacent in $\overline{G}$ if and only if they are non-adjacent in $G$. A \emph{matching} in $G$ is a set of pairwise disjoint edges.

For a vertex $u\in V(G)$, the set of neighbours of a vertex $u$ in $G$ is denoted by $N_{G}(u)$. For $U\subseteq V(G)$, the \emph{induced subgraph} $G[U]$ of $G$ satisfies $V(G[U])=U$ and $E(G[U])=\{uv\in E(G): u,v\in U\}$, and we say that the vertex subset $U$ induces the subgraph $G[U]$.
A graph $H$ is a \emph{subgraph} of a graph $G$, denoted by $H \subseteq G$,
if $V(H)\subseteq V(G)$ and $E(H)\subseteq E(G)$. We call $H$ a \emph{proper subgraph} of $G$, denoted by $H \subsetneq G$, if $H\subseteq G$ and $H \neq G$.
The \emph{union} of two graphs $H$ and $G$ is the graph $H\cup G$ with vertex set $V(H)\cup V(G)$ and edge set $E(H)\cup E(G)$.
For two disjoint graphs $H$ and $G$, the \emph{join} of $H$ and $G$, denoted by $H\vee G$, is the graph obtained from $H\cup G$ by adding all edges between $V(H)$ and $V(G)$.

We use $C_n$, $K_n$ and $K_{a,b}$ to denote the $n$-vertex cycle, the $n$-vertex complete graph and the complete bipartite graph with partitions of size $a$ and $b$, respectively. A cycle with $n$ vertices is called an $n$-cycle.
For a positive integer $n$, let $[n]=\{1,2,\dots,n\}$. For any terminology and notation that are not explicitly defined, we refer the reader to \cite{Bondy}.

The rest of this paper is organized as follows. In Section 2, we introduce some results that will be used in subsequent proofs and present the proofs of Theorems \ref{thm1-7} and \ref{thm1-8}. In Section 3, we prove Theorem \ref{thm1-9}. In Section 4, we conclude the paper.

\section{Proofs of Theorems \ref{thm1-7} and \ref{thm1-8}}
This section will be dedicated to proving Theorems \ref{thm1-7} and \ref{thm1-8}.
To this end, we first present several auxiliary results that will be used later.

\begin{theorem}[\upshape\cite{Hall}]\label{thm2-1}
Let $B$ be a bipartite graph with bipartition $(X,Y)$. There is a matching which covers every vertex in $X$ if and only if $|N_B(S)|\ge |S|$ for all $S\subseteq X$.
\end{theorem}

\begin{theorem}[\upshape\cite{Stanley}]\label{thm2-2}
For any graph $G$ of order $n$, we have $\rho^2(G)+\rho(G)\le2e(G)$, with equality if and only if $e(G)=\binom{k}{2}$ and $G\cong K_k\cup (n-k)K_1$ for some positive integer $k$.
\end{theorem}

\begin{corollary}\label{cor1}
Let $G$ be a graph of order $n$ satisfying $\rho(G)\ge n-2$. Then either
$G\cong K_{n-1}\cup K_1$ or $e(G)\ge \binom{n-1}{2}+1$.
\end{corollary}

\begin{proof}
Since $\rho(G)\ge n-2$, we have $e(G)\ge \binom{n-1}{2}$ by Theorem \ref{thm2-2}, and thus $e(G)\ge \binom{n-1}{2}+1$ or $e(G)=\binom{n-1}{2}$. If $e(G)=\binom{n-1}{2}$, then $\rho^2(G)+\rho(G)=2e(G)$, which implies $G\cong K_{n-1}\cup K_1$ by Theorem \ref{thm2-2}. This completes the proof.
\end{proof}

We introduce some notation that will be used throughout the subsequent proofs.

Let $n\ge 3$ and $\mathbf{G}=\{G_1, \dots, G_{n}\}$ be a collection of not necessarily distinct $n$-vertex graphs with a common vertex set $V$.
Let $G=\bigcup\limits_{i=1}^n G_i$ be a simple graph with $V(G)=V$ and $E(G)=\bigcup\limits_{i=1}^nE(G_i)$, and let $E(\overline{G})=E(K_n)\setminus E(G)=\bigcap\limits_{i=1}^n E(\overline{G_i})$ with $|E(\overline{G})|=m$.
For $3\le\ell\le n$, let $\mathcal{C}_\ell(G)$ be the set of all $\ell$-cycles in $G$ with $|\mathcal{C}_\ell(G)|=N_\ell$.
For an edge $e\in E(G)$, let $N_{e,\ell}=|\{C\in\mathcal{C}_\ell(G):e\in E(C)\}|$.

\begin{lemma}\label{lem2-1}
Suppose $N_\ell>0$ and $m\le n-3$. Then $N_{e,\ell}\le \frac{2N_\ell}{n-1-m}$ for every $e\in E(G)$.
\end{lemma}
\begin{proof}
Fix $e=x_0x_1\in E(G)$. The number of $\ell$-cycles not containing $e$ in $G$ is $N_\ell-N_{e,\ell}$, for which we now give a lower bound.

Firstly, for any $C=x_0x_1x_2\cdots x_{\ell-1}x_0\in \mathcal{C}_\ell(G)$, we define two types of cycle-switching operations as follows.

\vspace{6pt}
\noindent\textbf{Operation I:} The vertex set $V(C)$ remains unchanged. For each $j\in\{2,\ldots,\ell-2\}$, delete $x_0x_1$, $x_jx_{j+1}$, and add $x_0x_j$, $x_1x_{j+1}$.
\vspace{6pt}

\vspace{6pt}
\noindent\textbf{Operation II:} Replace $x_1$ with a vertex not belonging to $V(C)$. For each $z\in V\setminus V(C)$, delete $x_0x_1$, $x_1x_2$, and add $x_0z$, $zx_2$.
\vspace{6pt}

Secondly, we prove that for any $\ell$-cycle $C\in\mathcal{C}_\ell(G)$ with $e\in E(C)$, there exist at least $n-3-m$ cycle-switching operations such that each operation yields an $\ell$-cycle in $G$ that does not contain the edge $e$.

For Operation I, if the two added new edges $x_0x_j,x_1x_{j+1}\in E(G)$, then the result is the $\ell$-cycle $x_0x_jx_{j-1}\cdots x_1x_{j+1}x_{j+2}\cdots x_{\ell-1}x_0\in \mathcal{C}_\ell(G)$, which does not contain $e$. There are $\ell-3$ possible cycle-switching operations.

For Operation II, if the two added new edges $x_0z,zx_2\in E(G)$, the result is the $\ell$-cycle $x_0zx_2x_3\cdots x_{\ell-1}x_0\in \mathcal{C}_\ell(G)$ avoiding $e$. There are $n-\ell$ possible cycle-switching operations.

Thus there exist $(\ell-3)+(n-\ell)=n-3$ possible cycle-switching operations. Note that all edges added by these operations are pairwise distinct, which implies each edge in $E(\overline{G})$ can invalidate at most one cycle-switching operation. It follows that $C$ admits at least $n-3-m$ cycle-switching operations satisfying the requirement.

Thirdly, we prove that any $\ell$-cycle $C'\in\mathcal{C}_\ell(G)$ with $e\notin E(C')$ can be obtained from at most two $\ell$-cycles containing $e$ by cycle-switching operations I and II.

If $x_1\in V(C')$, then either $C'$ is obtained by Operation I, or it cannot be obtained by any of the cycle-switching operations I and II. Assume $C'$ arises from Operation I. There are two paths between $x_0$ and $x_1$ in $C'$, denoted by $P=x_0y_1\cdots y_rx_1$ and $Q=x_0z_1\cdots z_sx_1$. Then $C'$ can be obtained via Operation I from $C''=C'-\{x_0y_1,x_1z_s\}+\{x_0x_1,y_1z_s\}$ or $C'''=C'-\{x_0z_1,x_1y_r\}+\{x_0x_1,z_1y_r\}$.

If $x_1\notin V(C')$, then either $C'$ is obtained by Operation II, or it cannot be obtained by any of the cycle-switching operations I and II. Assume $C'$ arises from Operation II, and $N_{C'}(x_0)=\{u,v\}$. Then $C'$ can be obtained via Operation II from the cycle constructed by replacing $u$ or $v$ with $x_1$.

Finally, combining the above arguments, $G$ contains at least $\frac{(n-3-m)N_{e,\ell}}{2}$ $\ell$-cycles not containing $e$. It follows that $N_\ell-N_{e,\ell}\ge \frac{(n-3-m)N_{e,\ell}}{2}$, and thus $N_{e,\ell}\le \frac{2N_\ell}{n-1-m}$.

This completes the proof.
\end{proof}

\begin{lemma}\label{lem2-2}
If $e(G_i)\ge \binom{n-1}{2}+1$ for each $i\in [n]$, then $m\le n-2$, and $G$ is pancyclic if and only if $\mathbf{G}$ is rainbow pancyclic.
\end{lemma}

\begin{proof}
Since $e(G_i)\ge \binom{n-1}{2}+1$ for each $i\in [n]$, we have $e(\overline{G_i})\le n-2$. Recalling that $E(\overline{G})=E(K_n)\setminus E(G)=\bigcap\limits_{i=1}^n E(\overline{G_i})$ with $|E(\overline{G})|=m$, we have
$E(\overline{G})\subseteq E(\overline{G_i})$, and thus $m\le n-2$.

Now we show that $G$ is pancyclic if and only if $\mathbf{G}$ is rainbow pancyclic. The sufficiency is trivial, so it suffices to prove the necessity. Fix $\ell\in\{3,\dots,n\}$ and suppose $G$ contains an $\ell$-cycle. We only need to show there exists a rainbow $\ell$-cycle in $\mathbf{G}$.

For each $e\in E(G)$, we define $\omega(e)=|\{i\in[n]:e\notin E(G_i)\}|$. Then
\begin{equation}\label{e1}
\begin{aligned}
\sum_{e\in E(G)}\omega(e)
&=\sum_{i=1}^{n}|E(G)\setminus E(G_i)|\\
&=\sum_{i=1}^{n}|E(K_n)\setminus(E(\overline{G})\cup E(G_i))|\\
&=\sum_{i=1}^{n}|E(\overline{G_i})\setminus E(\overline{G})|\\
&\le n(n-2-m).
\end{aligned}
\end{equation}

%\begin{equation}\tag{1}\label{e1}
%\sum_{e\in E(G)}\omega(e)=\sum_{i=1}^n|E(\overline{G_i})\setminus D|\le n(n-2-d).
%\end{equation}

If $m=n-2$, then $E(\overline{G_i})=E(\overline{G})$ for each $i\in [n]$, and thus $G_1=\cdots=G_n=G$. Therefore, every cycle in $G$ is rainbow in $\mathbf{G}$.

If $m\le n-3$, for any $C\in\mathcal{C}_\ell(G)$, let $\omega(C)=\sum\limits_{e\in E(C)}\omega(e)$. Then
\begin{equation*}
\sum_{C\in\mathcal{C}_\ell(G)}\omega(C)=\sum_{C\in\mathcal{C}_\ell(G)}\sum_{e\in E(C)}\omega(e)=\sum_{e\in E(G)}\omega(e)|\{C\in\mathcal{C}_\ell(G):e\in E(C)\}|=\sum_{e\in E(G)}\omega(e)N_{e,\ell}.
\end{equation*}
%\begin{align*}
%   \sum_{C\in\mathcal{C}_\ell(G)}\omega(C)
%      &=\sum_{C\in\mathcal{C}_\ell(G)}\sum_{e\in E(C)}\omega(e)\\
%      &=\sum_{e\in E(G)}\omega(e)|C\in\mathcal{C}_\ell(G):e\in E(C)|\\
%      &=\sum_{e\in E(G)}\omega(e)N_{e,\ell}.
%\end{align*}
By \eqref{e1}, Lemma \ref{lem2-1} and $2(n-1)-\frac{2n(n-2-m)}{n-1-m}=\frac{2(m+1)}{n-1-m}>0$, we have
\begin{equation}\label{2}
   \sum_{C\in\mathcal{C}_\ell(G)}\omega(C)\le \frac{2N_\ell}{n-1-m}\sum_{e\in E(G)}\omega(e)\le \frac{2n(n-2-m)}{n-1-m}N_\ell<2(n-1)N_\ell.
\end{equation}
Since $|\mathcal{C}_\ell(G)|=N_\ell$, \eqref{2} implies that there is a cycle $C'\in\mathcal{C}_\ell(G)$ such that
\begin{equation}\label{e2}
\omega(C')<2(n-1).
\end{equation}

Let $E(C')=\{f_1,f_2,\dots,f_\ell\}$. We construct a bipartite graph $B$ with bipartition $(E(C'),[n])$ such that $V(B)=E(C')\cup [n]$ and $fi\in E(B)$ if and only if $f\in E(G_i)$, where $f\in E(C')$ and $i\in[n]$. Since $E(C')\subseteq E(G)$, we have $d_B(f)\ge 1$ for any $f\in E(C')$.
Clearly, a matching in $B$ which covers all vertices in $\{f_1,f_2,\dots,f_\ell\}$ corresponds to a rainbow coloring of $C'$.

If there is no such matching in $B$, by Theorem \ref{thm2-1}, there exists a nonempty set $S\subseteq E(C')$ such that $|N_B(S)|\le |S|-1$. It follows that $|S|\geq2$ since $d_B(f)\ge 1$ for any $f\in E(C')$ implies $|N_B(S)|\ge 1$.
Let $Q=[n]\setminus N_B(S)$. Then $|Q|\ge n-|S|+1$.

On the other hand, for each $f\in S\subseteq E(C')$ and each $i\in Q$, we have $fi\notin E(B)$, say, $f\in E(\overline{G_i})$, thus $S\subseteq E(\overline{G_i})$ for each $i\in Q$, which implies $\omega(C')\ge |S||Q|$, and $|S|\le n-2$ since $e(\overline{G_i})\le n-2$ for each $i\in Q$.
Therefore,
\[\omega(C')\ge |S||Q|\ge |S|(n-|S|+1)=2(n-1)+(|S|-2)(n-|S|-1)\ge2(n-1),\]
contradicting \eqref{e2}. Thus the bipartite graph $B$ has a matching which covers all vertices in $\{f_1,f_2,\dots,f_\ell\}$, which implies $C'\in \mathcal{C}_\ell(G)$ is rainbow. It follows that $\mathbf{G}$ contains a rainbow $\ell$-cycle, and thus $\mathbf{G}$ is rainbow pancyclic.
\end{proof}

We now prove Theorems \ref{thm1-7} and \ref{thm1-8}.

\begin{proof}[\noindent\textbf{Proof of Theorem \ref{thm1-7}}]
Recalling that $G=\bigcup\limits_{i=1}^n G_i$, then $e(G)\ge e(G_i)\ge \binom{n-1}{2}+1$. By Theorem \ref{thm1-4}, $G$ is pancyclic unless $G\in\{K_1\vee(K_{n-2}\cup K_1),K_2\vee 3K_1,K_{2,2}\}$.

If $G$ is pancyclic, then $\mathbf{G}$ is rainbow pancyclic by Lemma \ref{lem2-2}.

If $G\in\{K_1\vee(K_{n-2}\cup K_1),K_2\vee 3K_1,K_{2,2}\}$, then $e(G)=\binom{n-1}{2}+1$. Since $G_i\subseteq G$ and $e(G_i)\ge \binom{n-1}{2}+1$, we have $G_i=G$ for each $i\in [n]$, which implies that $G_1=G_2=\dots=G_n\cong K_1\vee(K_{n-2}\cup K_1)$ or $G_1=G_2=\dots=G_5\cong K_2\vee 3K_1$ with $n=5$ or $G_1=G_2=G_3=G_4\cong K_{2,2}$ with $n=4$. Clearly, $G$ is not pancyclic, and thus $\mathbf{G}$ is not rainbow pancyclic.

Therefore, the proof of Theorem \ref{thm1-7} is completed.
\end{proof}

\begin{proof}[\noindent\textbf{Proof of Theorem \ref{thm1-8}}]
Since $\rho(G_i)> n-2$ for each $i\in [n]$, by Theorem \ref{thm2-2}, we have $e(G_i)\ge \binom{n-1}{2}+1$ for each $i\in [n]$. By Theorem \ref{thm1-7}, $\mathbf{G}$ is rainbow pancyclic unless one of the following holds:
(i) $G_1=G_2=\dots=G_n\cong K_1\vee(K_{n-2}\cup K_1)$; (ii) $n=5$ and $G_1=G_2=\dots=G_5\cong K_2\vee 3K_1$; (iii) $n=4$ and $G_1=G_2=G_3=G_4\cong K_{2,2}$.

Clearly, $\rho(K_1\vee(K_{n-2}\cup K_1))>\rho(K_{n-1}\cup K_1)=n-2$, $\rho(K_2\vee 3K_1)=3=n-2$ and $\rho(K_{2,2})=2=n-2$. Therefore, if $\rho(G_i)> n-2$ for each $i\in [n]$, then $\mathbf{G}$ is rainbow pancyclic unless $G_1=G_2=\dots=G_n\cong K_1\vee(K_{n-2}\cup K_1)$.
\end{proof}

\section{Proof of Theorem \ref{thm1-9}}
In this section, we show Theorem \ref{thm1-9}.

Let $\mathbf{G}$, $G$, $E(\overline{G})$, $m$, $\mathcal{C}_\ell(G)$, and $N_\ell$ be as defined in Section 2.
Since $\rho(G_i)\geq n-2$ for each $i\in [n]$, by Corollary \ref{cor1}, we have $G_i\cong K_{n-1}\cup K_1$ or $e(G_i)\ge \binom{n-1}{2}+1$ for any $i\in[n]$.
Now we define the sets $R$ and $T$ as follows:
\[
R=\{i\in [n]: G_i\cong K_{n-1}\cup K_1\},\quad T=[n]\setminus R.
\]
Let $|R|=r$ and $|T|=t$. Then $r+t=n$.

Clearly, for each $i\in T$, $e(G_i)\ge\binom{n-1}{2}+1$ and $e(\overline{G_i})\le n-2$. Recalling the definition of $F_v$, for each $i\in R$, there exists a vertex $v\in V$ such that $G_i=F_v$, for convenience, we write $v=v_i$. In particular, if $G_i=G_j$ for $i,j\in R$, we set $v_i=v_j$.

For each $e\in E(G)$, we define $\omega_T(e)=|\{i\in T:e\notin E(G_i)\}|$. For a subgraph $H$ of $G$, let $ \omega_T(H)=\sum\limits_{e\in E(H)}\omega_T(e)$.
Similar to the proof of Lemma \ref{lem2-2}, we have the following results.

\begin{lemma}\label{lem4-1}
Let $t>0$ and $N_\ell>0$ for $3\le \ell\le n$. Then $m\le n-2$, and there exists a cycle $C\in\mathcal{C}_\ell(G)$ such that $\omega_T(C)<2t$.
\end{lemma}
\begin{proof}
Clearly, $T\neq\emptyset$ as $t>0$.
Recalling that $E(\overline{G})=E(K_n)\setminus E(G)=\bigcap\limits_{i=1}^n E(\overline{G_i})$ with $|E(\overline{G})|=m$, we have $E(\overline{G})\subseteq E(\overline{G_i})$ for any $i\in T$, and thus $m\le e(\overline{G_i})\le n-2$.

By the definition of $\omega_T(e)$ and $|T|=t$, we have
\begin{equation}\label{e3}
\sum_{e\in E(G)}\omega_T(e)=\sum_{i\in T}|E(\overline{G_i})\setminus E(\overline{G})|\le t(n-2-m).
\end{equation}

%\begin{equation}\label{e3}
%\begin{aligned}
%\sum_{e\in E(G)}\omega_T(e)
%&=\sum_{i\in T}|E(G)\setminus E(G_i)|\\
%&=\sum_{i\in T}|E(K_n)\setminus (D\cup E(G_i))|\\
%&=\sum_{i\in T}(e(\overline{G_i})-d)\\
%&\le t(n-2-d).
%\end{aligned}
%\end{equation}

If $m=n-2$, then $e(\overline{G_i})=n-2$ and $G_i=G$ for each $i\in T$,
$\omega_T(C)=\sum\limits_{e\in E(C)}\omega_T(e)\le \sum\limits_{e\in E(G)}\omega_T(e)\le 0<2t$ by \eqref{e3} for any cycle $C$ in $G$, and thus the conclusion follows.

%for every $i\in T$, $E(\overline{G_i})=D$, which implies $G_i=G$. Thus $\omega_T(e)=0$ for each $e\in E(G)$, so the conclusion follows.

If $m\le n-3$, by Lemma \ref{lem2-1} and \eqref{e3}, we have
\[
\begin{aligned}
\sum_{C\in\mathcal{C}_\ell(G)}\omega_T(C)
&=\sum_{C\in\mathcal{C}_\ell(G)}\sum_{e\in E(C)}\omega_T(e)
=\sum_{e\in E(G)}\omega_T(e)|\{C\in\mathcal{C}_\ell(G): e\in E(C)\}|\\
&=\sum_{e\in E(G)}\omega_T(e)N_{e,\ell}
\le\frac{2N_\ell}{n-1-m}\sum_{e\in E(G)}\omega_T(e)
\le\frac{2t(n-2-m)}{n-1-m}N_\ell\\
&<2tN_\ell.
\end{aligned}
\]
%\begin{align*}
%   \sum_{C\in\mathcal{C}_\ell(G)}\omega_T(C)
%      &=\sum_{C\in\mathcal{C}_\ell(G)}\sum_{e\in E(C)}\omega_T(e)\\
%      &=\sum_{e\in E(G)}\omega_T(e)|\{C\in\mathcal{C}_\ell(G): e\in E(C)\}|\\
%      &=\sum_{e\in E(G)}\omega_T(e)N_{e,\ell}\\
%      &\le\frac{2N_\ell}{n-1-m}\sum_{e\in E(G)}\omega_T(e)\\
%      &\le\frac{2t(n-2-m)}{n-1-m}N_\ell\\
%      &<2tN_\ell\\
%      &=2t|\mathcal{C}_\ell(G)|.
%\end{align*}
It follows that there exists a cycle $C\in\mathcal{C}_\ell(G)$ such that $\omega_T(C)<2t$.

This completes the proof.
\end{proof}

The following basic observation follows easily from Hall's theorem (i.e. Theorem \ref{thm2-1}), and we omit its proof.

\begin{observation}\label{ob4-1}
Let $0< a\le b$. If a bipartite graph is obtained from $K_{a,b}$ by deleting at most one edge, then it has a matching covering all vertices in the partition of size $a$, except when $a=b=1$ and the only edge is removed.
\end{observation}

Next, we give a key lemma.
\begin{lemma}\label{lem4-2}
Let $\mathbf{G}$ be a graph collection satisfying the conditions of Theorem \ref{thm1-9}. If $N_\ell>0$ for $3\le\ell\le n$, then one of the following holds:
\begin{enumerate}[label=(\roman*), font=\upshape, itemsep=2pt, align=left, leftmargin=1em]
 \item $\mathbf{G}$ contains a rainbow $\ell$-cycle;
 \item $\ell=n$ and there exists a vertex $v\in V$ such that $|\{i\in[n]:G_i=F_v\}|=n-1$.
\end{enumerate}
\end{lemma}

\begin{proof}
For any cycle $C$ in $G$, similar to the proof of Lemma \ref{lem2-2}, we construct a bipartite graph $B_{C}$ with bipartition $(E(C),[n])$, where $e\in E(C)$ is adjacent to $i\in[n]$ if and only if $e\in E(G_i)$. Since $E(C)\subseteq E(G)$, we have $d_{B_{C}}(e)\ge 1$ for any $e\in E(C)$.

\vspace{6pt}
$\mathbf{Case~1.}$ There exists some $C\in\mathcal{C}_\ell(G)$ such that $B_{C}$ contains a matching covering all vertices in $E(C)$.
\vspace{6pt}

Since a matching in $B_{C}$ which covers all vertices in $E(C)$ corresponds to a rainbow coloring of $C$, it follows that $\mathbf{G}$ contains a rainbow $\ell$-cycle, and thus (i) holds.

\vspace{6pt}
$\mathbf{Case~2.}$ There is no such matching for any $C\in\mathcal{C}_\ell(G)$.
\vspace{6pt}

By Theorem \ref{thm2-1}, for any $C\in\mathcal{C}_\ell(G)$, there is a nonempty set $S\subseteq E(C)$ such that $|N_{B_{C}}(S)|\le |S|-1\le \ell-1\le n-1$. Since $|N_{B_{C}}(S)|>0$, we have $|S|\ge 2$. Let $Q=[n]\setminus N_{B_{C}}(S)$. Then $|Q|\ge n-|S|+1$.

Recall the sets $R$ and $T$ defined above, where $0\le |T|=t\le n$.
Now we consider the following three subcases according to the size of $t$.

\vspace{6pt}
$\mathbf{Subcase~2.1.}$ $t=n$.
\vspace{6pt}

It is clear that $e(\overline{G_i})\le n-2$ for each $i\in [n]$. Since $N_\ell>0$, $G$ contains an $\ell$-cycle, and it follows from the necessity proof of Lemma \ref{lem2-2} that $\mathbf{G}$ contains a rainbow $\ell$-cycle. Thus (i) holds.

\vspace{6pt}
$\mathbf{Subcase~2.2.}$ $t=0$.
\vspace{6pt}

It immediately follows that $r=n$ and $G_i=F_{v_i}$ for some $v_i\in V$ and any $i\in[n]$. Choose a cycle $C\in\mathcal{C}_\ell(G)$, and let $S$ and $Q$ be defined as above.
Now we show $|S|=2$.

If $|S|\ge 3$, then $|E(G_i)\cap S|\ge 1$ for any $i\in[n]$ since $G_i=F_{v_i}$ for some $v_i\in V$ and $S\subseteq E(C)$. This implies $|N_{B_{C}}(S)|=n$, a contradiction. Thus $|S|\le 2$, and so $|S|=2$.

Since $|N_{B_{C}}(S)|>0$ and $|N_{B_{C}}(S)|\le |S|-1=1$, we have $|N_{B_{C}}(S)|=1$, and thus $|Q|=n-1$, which implies there are $n-1$ graphs in $\mathbf{G}$ missing the two edges of $S$ and $G_k$ contains the two edges of $S$, where $\{k\}=N_{B_{C}}(S)$. Since $G_i=F_{v_i}$ for some $v_i\in V$ for any $i\in[n]$, we can obtain that there exists a vertex $v\in V$ such that $|\{i\in[n]:G_i=F_v\}|=n-1$. Without loss of generality, we take $G_1=G_2=\dots=G_{n-1}=F_v$.

If $\ell=n$, then (ii) holds. If $3\le\ell\le n-1$, then there exists a rainbow $\ell$-cycle in $\{G_1, \dots, G_{n-1}\}$, and thus (i) holds.

\vspace{6pt}
$\mathbf{Subcase~2.3.}$ $1\le t\le n-1$.
\vspace{6pt}

Since $r+t=n$, we have $1\le r\le n-1$.
By Lemma \ref{lem4-1}, there exists a cycle $C\in\mathcal{C}_\ell(G)$ such that $\omega_T(C)<2t$. Let $S$ and $Q$ be defined as above. We can obtain the following claims by the property of $C$.

\begin{claimnum}\label{claim1}
$Q\cap R\neq\emptyset$.
\end{claimnum}

\begin{proof}
If not, then $Q\subseteq T$. Thus $|Q|\le t$, and $e(\overline{G_i})\le n-2$ for each $i\in Q$ by the definition of $T$. Since $N_{B_{C}}(S) \cap Q=\emptyset$, we have $e(\overline{G_i})\ge |S|$ for each $i\in Q$. It follows that $2\le |S|\le n-2$, which implies $n\ge 4$ and $3\le n-|S|+1\le |Q|\le t$.

By the definition of $\omega_T(C)$ and the fact that $Q\subseteq T$, we have

\[\omega_T(C)=\sum\limits_{e\in E(C)}|\{i\in T:e\notin E(G_i)\}|\ge |S||Q|\ge |S|(n-|S|+1).\]
Let $F(x)=x(n-x+1)$, where $3\le x\le t$. Then $F(3)=3(n-2)\ge 2(n-1)\ge 2t$, and $F(t)=t(n-t+1)=t(r+1)\ge 2t$ as $r+t=n$. Thus $F(x)\ge \min\{F(3),F(t)\}\geq 2t$ for $3\le x\le t$. Therefore, $\omega_T(C)\ge |S|(n-|S|+1)=F(n-|S|+1)\ge 2t$ since $3\le n-|S|+1\le t$. This contradicts $\omega_T(C)<2t$.
\end{proof}

By Claim \ref{claim1}, we take $j\in Q\cap R$. Then $G_j=F_{v_j}$ for some $v_j\in V$, and thus $|E(C)\setminus E(G_j)|=\begin{cases}
0, & \text{if } v_j\notin V(C);\\
2, & \text{if } v_j\in V(C).
\end{cases}$ Combining this with $j\in Q$, we have $S\subseteq E(C)\cap E(\overline{G_j})$, then $|S|=2$ as $|S|\ge2$, and thus $|N_{B_{C}}(S)|=1$.
Without loss of generality, we take $S=\{v_ju,v_jw\}$ and $N_{B_{C}}(S)=\{k\}$. Then $S\subseteq E(G_k)$.

\begin{claimnum}\label{claim2}
For each $i\in R$, $G_i=G_j=F_{v_j}$.
\end{claimnum}

\begin{proof}
Assume for contradiction that there exists some $i\in R$ such that $G_{i}=F_{v_{i}}$, where $v_{i}\neq v_j$.

Now we consider $v_i$ and show $i=k$. In fact, if $v_{i}=u$, then $v_jw\in E(G_{i})$; if $v_{i}=w$, then $v_ju\in E(G_{i})$; if $v_{i}\notin\{u,w\}$, then $\{v_ju,v_jw\}\subseteq E(G_{i})$. Thus $i\in N_{B_{C}}(S)$, which implies $i=k$. It follows that $T\subseteq Q$. Therefore, $\omega_T(C)=\sum\limits_{e\in E(C)}|\{i\in T:e\notin E(G_i)\}|\ge |S||T|=2t$. This contradicts $\omega_T(C)<2t$.
\end{proof}

Since $S\cap E(G_j)=\emptyset$ and $S\subseteq E(G_k)$, by Claim \ref{claim2}, we have $k\in T$. Thus $\omega_T(v_ju)=\omega_T(v_jw)=t-1$. Let $P=C-\{v_j\}=x_1x_2\cdots x_{\ell-1}$, where $x_1=u$ and $x_{\ell-1}=w$. Then
\begin{equation}\label{e4}
\omega_T(P)=\omega_T(C)-\omega_T(v_ju)-\omega_T(v_jw)<2t-2(t-1)=2.
\end{equation}

$\mathbf{Subcase~2.3.1.}$ $t=1$.

In this subcase, $r=n-1$. By Claim \ref{claim2}, $|\{i\in[n]:G_i=F_{v_j}\}|=n-1$. If $\ell=n$, then (ii) holds. If $3\le\ell\le n-1$, then there exists a rainbow $\ell$-cycle in $\mathbf{G}\setminus\{G_k\}$, and thus (i) holds.

$\mathbf{Subcase~2.3.2.}$ $2\le t\le n-1$.

Let $p\in T\setminus\{k\}$. Then $p\in Q\cap T$ since $N_{B_{C}}(S)=\{k\}$, which implies $S\cap E(G_p)=\emptyset$ and $e(G_p)\ge\binom{n-1}{2}+1$, and thus $u,w\notin N_{G_p}(v_j)$ and $d_{G_p}(v_j)\ge 1$. Therefore, $n\ge 4$, and we take $z\in N_{G_p}(v_j)$.

Now we construct a path $P'$ of order $\ell-1$ in $G$ from $x_1(=u)$ to $z$ such that $P'$ differs from $P$ in exactly one edge.
If $z=x_a\in V(P)$, then $x_a\notin\{u,w\}$ as $z=x_a\in N_{G_p}(v_j)$, which implies $2\le a\le \ell-2$, and thus we take $P'=x_1x_2\cdots x_{a-1}x_{\ell-1}x_{\ell-2}\cdots x_a$ since $x_{a-1}x_{\ell-1}\in E(G_i)$ for each $i\in R$ by Claim \ref{claim2}.
If $z\notin V(P)$, then we take $P'=x_1x_2\cdots x_{\ell-2}z$ since $x_{\ell-2}z\in E(G_i)$ for each $i\in R$ by Claim \ref{claim2}. Thus $C'=v_jx_1P'zv_j\in\mathcal{C}_\ell(G)$. Now we show that $C'$ is a rainbow $\ell$-cycle.

Let $e'$ be the unique edge in $E(P')\setminus E(P)$ and $q\in R$. Then $e'\in E(G_q)$ by Claim \ref{claim2}. Note that $v_jx_1\in E(G_k)$ and $v_jz\in E(G_p)$, where $k,p\in T$.
By the construction of $P',C'$, we see that $E(C')\setminus\{v_jx_1,v_jz,e'\}=E(P)\cap E(P')$.

If $\ell=3$, then $E(P)\cap E(P')=\emptyset$, $C'=v_jx_1zv_j$ and $e'=x_1z$. It follows that $C'$ is a rainbow 3-cycle with $v_jx_1\in E(G_k)$, $x_1z\in E(G_q)$ and $v_jz\in E(G_{p})$, and thus (i) holds.

If $\ell\ge4$, we construct a bipartite graph $B'$ with bipartition $(E(P)\cap E(P'),[n]\setminus\{k,p,q\})$, where $e\in E(P)\cap E(P')$ is adjacent to $i\in[n]\setminus\{k,p,q\}$ if and only if $e\in E(G_i)$. By Claim \ref{claim2}, we have $ei\in E(B')$ for any $e\in E(P)\cap E(P')$ and any $i\in R\cap ([n]\setminus\{k,p,q\})$. Combining this with \eqref{e4}, it follows that $B'$ is a graph obtained from $K_{\ell-3,n-3}$ by deleting at most one edge.

By Observation \ref{ob4-1}, $B'$ has a matching covering all vertices in $E(P)\cap E(P')$, except when $\ell-3=n-3=1$ and the only edge is removed. If $B'$ contains such a matching, then $C'$ is a rainbow $\ell$-cycle, and thus (i) holds.

In the following, we consider $\ell-3=n-3=1$, i.e., $n=\ell=4$. Thus $P=uzw$, $P'=uwz$, $C'=v_juwzv_j$ and $e'=uw$. From earlier arguments, we know $v_ju\in E(G_k)$, $v_jz\in E(G_p)$ and $uw\in E(G_q)$, where $k,p\in T$ and $q\in R$. Let $\{p'\}=[4]\setminus\{k,p,q\}$. Then $p'\in Q$ and thus $\{v_ju,v_jw\}\cap E(G_{p'})=\emptyset$.

If $p'\in R$, then $wz\in E(G_{p'})$. If $p'\in T$, then $e(\overline{G_{p'}})\le n-2$, and thus $wz\in E(G_{p'})$ since $\{v_ju,v_jw\}\cap E(G_{p'})=\emptyset$. Therefore, $C'$ is a rainbow 4-cycle with $v_ju\in E(G_k)$, $uw\in E(G_q)$, $wz\in E(G_{p'})$ and $zv_j\in E(G_p)$, and thus (i) holds.

Combining the above arguments, either (i) or (ii) holds, and thus the proof is completed.
\end{proof}

Now we give the proof of Theorem \ref{thm1-9}.
\begin{proof}[\noindent\textbf{Proof of Theorem \ref{thm1-9}}]
We first prove the sufficiency.
If (i) or (ii) holds, then $v$ is not contained in any rainbow cycle, and thus $\mathbf{G}$ contains no rainbow $n$-cycle. If (iii) holds, then $\mathbf{G}$ contains no rainbow $5$-cycle. If (iv) holds, then $\mathbf{G}$ contains no rainbow $3$-cycle. Thus $\mathbf{G}$ is not rainbow pancyclic.

Next, we prove the necessity. Assume that $\mathbf{G}$ is not rainbow pancyclic. Since $\rho(G_i)\ge n-2$ for each $i\in [n]$, by Theorem \ref{thm2-2}, we have $e(G_i)\ge \binom{n-1}{2}$ for each $i\in [n]$, and thus $e(G)\ge \binom{n-1}{2}$. We distinguish two cases according to $e(G)$.

\vspace{6pt}
$\mathbf{Case~1.}$ $e(G)=\binom{n-1}{2}$.
\vspace{6pt}

Since $G_i\subseteq G$ and $e(G_i)\ge \binom{n-1}{2}=e(G)$ for each $i\in [n]$, then $G_1=G_2=\dots=G_{n}=G$. By Corollary \ref{cor1}, $G=F_v\cong K_{n-1}\cup K_1$ for some $v\in V$, which implies $|\{i\in[n]:G_i=F_v\}|=n$. Thus both (i) and (ii) hold.

\vspace{6pt}
$\mathbf{Case~2.}$ $e(G)\ge \binom{n-1}{2}+1$.
\vspace{6pt}

By (i) of Theorem \ref{thm1-4}, $G$ is pancyclic unless $G\in\{K_1\vee(K_{n-2}\cup K_1),K_2\vee 3K_1,K_{2,2}\}$.

If $G$ is pancyclic, then by Lemma \ref{lem4-2}, either $\mathbf{G}$ contains a rainbow $\ell$-cycle for each $3\le\ell\le n$, or there exists a vertex $v\in V$ such that $|\{i\in[n]:G_i=F_v\}|=n-1$. Since $\mathbf{G}$ is not rainbow pancyclic, (i) holds.

If $G\cong K_1\vee(K_{n-2}\cup K_1)$, then there exist distinct $u,v\in V$ such that $G_i\subseteq G=A_{v,u}$ for each $i\in [n]$. For $n=3$, we have $G_i\in\{F_v,F_w,A_{v,u}\}$ for each $i\in [3]$, where $V=\{u,v,w\}$.
For $n\ge 4$, if $G_i\neq A_{v,u}$ for some $i\in [n]$, then $e(G_i)<e(A_{v,u})=\binom{n-1}{2}+1$. Since $\rho(G_i)\ge n-2$, by Corollary \ref{cor1}, we have $G_i=F_{v_i}$ for some $v_i\in V$. It follows that $G_i=F_{v}$ since $G_i\subseteq A_{v,u}$ and $n\ge 4$. Therefore, $G_i\in\{F_v,A_{v,u}\}$ for $n\ge 4$ and each $i\in [n]$. Thus (ii) holds.

If $G\cong K_2\vee 3K_1$, then $n=5$ and $\rho(G)=3=n-2$. If $G_i\subsetneq G$ for some $i\in [5]$, then $\rho(G_i)<\rho(G)=n-2$, a contradiction. Thus $G_i=G$ for each $i\in [5]$, and so (iii) holds.

If $G\cong K_{2,2}$, then $n=4$ and $\rho(G)=2=n-2$. Similarly, $G_i=G$ for each $i\in [4]$, and thus (iv) holds.

Therefore, the proof of Theorem \ref{thm1-9} is completed.
\end{proof}

\section{Conclusion}
In this paper, we give edge and spectral conditions to guarantee that a graph collection is rainbow pancyclic. Our results generalize the classical results for graphs to graph collections and support Conjecture \ref{con1-1} in the context of graph collections.

In recent years, there has been extensive research on spectral conditions that guarantee the existence of rainbow subgraphs in graph collections, see, \cite{Chen,Chen2026,Guo,Guo2026,He,He2026,Zhang2026} and so on. It is worth noting that almost all existing studies concerning spectral conditions in graph collections rely on a shifting technique, i.e. the Kelmans operation. In this paper, we adopt a new approach and give our proof from a structural perspective.

\section*{Funding}
This work is supported by the National Natural Science Foundation of China (Grant Nos. 12371347, 12271337).

\section*{Declarations}
\noindent\textbf{Conflict of interest} The authors declare that they have no known competing financial interests or personal relationships that could have appeared to influence the work reported in this paper.\\
\textbf{Data availability} No data was used for the research described in the article.

\end{spacing}
\end{document}